\documentclass{article}

\usepackage{amsmath,amsfonts,amsthm,amssymb,graphics,graphicx}

\newtheorem{thm}{Theorem}[section]
\newtheorem{defin}[thm]{Definition}

\newtheorem{lem}[thm]{Lemma}

\newtheorem{prop}[thm]{Proposition}
\newtheorem{remark}[thm]{Remark}

\newcommand{\norm}[1]{\left|\!\left|{#1}\right|\!\right|}
\newcommand{\R}{\ensuremath{\mathbb{R}}}

\title{Semiclassical $L^{p}$ Estimates of Quasimodes on Submanifolds}
\author{Melissa Tacy}

\begin{document}

\maketitle

\begin{abstract}
Let $P = P(h)$  be a semiclassical pseudodifferential operator on a Riemannian manifold $M$. Suppose that $u(h)$ is a localised, $L^{2}$ normalised family of functions such that $P(h)u(h)$ is  $O(h)$  in $L^{2}$, as $h \to 0$. Then, for any submanifold $Y \subset M$, we obtain estimates on the $L^p$ norm of $u(h)$ restricted to $Y$, with exponents that are sharp for $h \to 0$. These results generalise those of Burq, G\'{e}rard and Tzvetkov \cite{burq} on $L^{p}$ norms for restriction of Laplacian eigenfunctions.
As part of the technical development we prove some extensions of the abstract Strichartz estimates of Keel and Tao \cite{keel}.
\end{abstract}

Let $P = P(h)$ be a semiclassical pseudodifferential operator on a Riemannian manifold $M$. We will assume that $P$ has a real principal symbol, and that its full symbol is smooth in the semiclassical parameter $h$. Other more technical assumptions on $P$ are given in Definition \ref{admissable}.
We prove estimates for approximate solutions $u = u(h)$ to the equation $P(h) u(h) = 0$. As usual in semiclassical analysis we assume that $u(h)$ is defined at least for a sequence $h_n$ tending to zero. 

Our precise definition of approximate solution, or quasimode, is that $P(u) = O_{L^{2}}(h)$ as $h \to 0$. This definition is natural with respect to localisation: if $P(u) = O_{L^{2}}(h)$, and $\chi$ is a pseudodifferential operator of order zero (with a symbol smooth in $h$), then $P(\chi u)$ is also $O_{L^{2}}(h)$. We will make the assumption that $u(h)$ can be localised, see Definition \ref{local}, and therefore will be able to reduce the problem to one of local analysis.

Given a submanifold $Y$ of $M$, we estimate the $L^p$ norm of the restriction of $u$ to $Y$, assuming the normalisation condition $\| u \|_{L^2(M)} = 1$. These estimates are of the form $\| u \|_{L^p(Y)} \leq C h^{-\delta}$ where $\delta$ depends on the dimension $n$ of $M$, the dimension $k$ of $Y$ and $p$ (except for one case where there is a logarithmic divergence) --- see Theorem \ref{maintheorem}. In every case the exponent $\delta(n, k, p)$ given by Theorem \ref{maintheorem} is optimal. Figure ~\ref{hypersurface} shows the exponent $\delta$ for a hypersurface  and, for comparison, the $L^{p}$ estimates over the whole manifold (Sogge \cite{sogge} for spectral clusters and Koch-Tataru-Zworski \cite{koch} for semiclassical operators). Figure ~\ref{submanifold} shows $\delta(n,k,p)$ for submanifolds of codimension greater than one.

\begin{figure}
\centering
\includegraphics[scale=0.45]{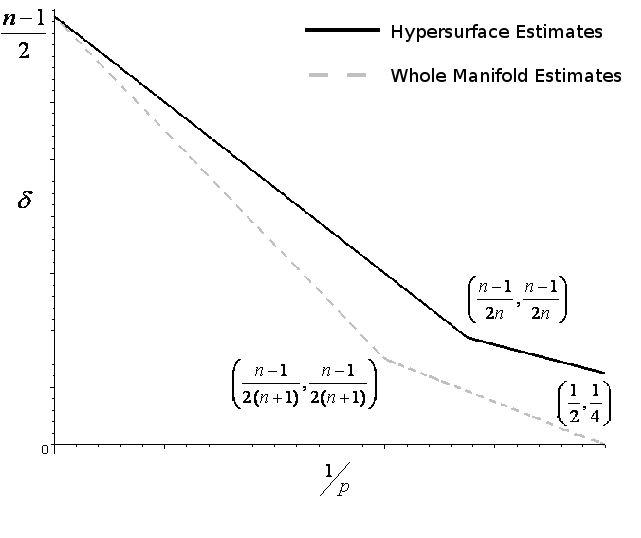} 
\caption{$\delta(n,k,p)$ plotted against $1/p$ for the hypersurface case (solid line) and whole manifold estimates (dashed line)}
\label{hypersurface}
\end{figure}

\begin{figure}
\centering
\includegraphics[scale=0.6]{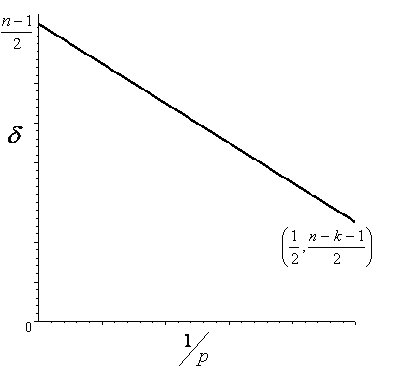} 
\caption{$\delta(n,k,p)$ plotted against $1/p$ for submanifolds of codimension greater than one}
\label{submanifold}
\end{figure}

The potential growth/concentration  of the quasimodes of a semiclassical operator is of great interest due to the connection to Quantum Mechanics. It is from Quantum Mechanics that we get the important set of motivating examples,
\begin{equation}P=h^{2}\Delta_{g}+V_{0}(x)\label{semiop}\end{equation}
here $\Delta_{g}$ is the (positive) Laplace-Beltrami operator associated with the metric $g$. We can transition between this picture and the usual eigenfunction picture of Quantum Mechanics by dividing the eigenfunction equation
$$\Delta{}u+V_{0}(x)u=Eu$$
by $E$. Then setting $E=1/h^{2}$ we have 
$$h^{2}\Delta{}u+h^{2}V_{0}(x)u-u=0$$
or 
$$Pu=0$$
where $P$ is as in $(\ref{semiop})$ with a potential term of $h^{2}V_{0}(x)-1$. Therefore the higher eigenvalue asymptotics of eigenfunctions of Quantum Mechanical systems corresponds to the $h\to{}0$ limit in semiclassical analysis. When $V_{0}(x)=0$ this problem reduces to estimating the size of Laplacian eigenfunctions restricted to a submanifold. A complete set of estimates for Laplacian eigenfunctions on compact manifolds is given by Burq, G\'{e}rard and Tzvetkov \cite{burq}.

A number of different techniques for studying the potential concentrations of eigenfunctions are available. A large body of recent work focuses on semiclassical measures (see for example Anatharaman \cite{anantharaman}, G\'{e}rard-Leichtnam \cite{gerard}, Zelditch \cite{zelditch} and Zelditch-Zworski \cite{zelditch2}). Sogge's work \cite{sogge} on spectral clusters give estimates for $\norm{u}_{L^{p}(M)}$ of the form $\norm{u}_{L^{p}(M)}\lesssim\lambda^{-\delta(n,p)}$ where $\lambda$ is the eigenvalue of $u$. This work is extended into the semiclassical regime by Burq, G\'{e}rard and Tzvetkov \cite{burq3} (for the Laplacian) and Koch, Tataru and Zworski \cite{koch} (for semiclassical operators). Multilinear estimates for spherical harmoics have also been obtained by Burq, G\'{e}rard and Tzvetkov \cite{burq3} In related work Koch and Tataru \cite{koch2} give $L^{p}$ estimates for eigenfunctions of the Hermite operator $H=-\Delta+x^{2}$ . In 2004 Reznikov \cite{reznikov} proved bounds for restrictions of Laplacian eigenfunctions to curves where the underlying manifold was a hyperbolic surface. In 2007 Burq, G\'{e}rard and Tzvetkov \cite{burq} produced results giving $L^{p}$ estimates of the restriction of eigenfunctions of the Laplace-Beltrami  operator on a compact manifold to a submanifold. This work directly extends these results using techniques found in Koch-Tataru-Zworski \cite{koch} and Burq-G\'{e}rard-Tzvetkov \cite{burq2} to move them into the more general semiclassical setting.

To continue we must define some objects from Semiclassical Analysis and give some basic results. A more detailed discussion of Semiclassical Analysis can be found in \cite{koch}, \cite{evans} and \cite{burq2}, however for the reader's convenience the main definitions and results used in this paper are provided in Section \ref{semiclassical}.

\section*{Acknowledgements}
I would like to thank Maciej Zworski for encouraging me to work on this problem and for many helpful conversations. Many thanks also to my supervisor Andrew Hassell for all his help preparing this manuscript and to Patrick G\'{e}rard for pointing out some useful references. While writing this paper I was supported by an Australian Postgraduate Award and much of the mathematical work was done while at the University of California Berkeley supported by a Fulbright Scholarship. My thanks to the Berkeley Mathematics Department for their hospitality.

\section{Semiclassical Analysis}\label{semiclassical}

Semiclassical analysis allows us to study Pseudodifferential and Fourier Integral Operators depending on a parameter which we denote as $h$. We think of this parameter as being small and obtain error terms bounded by powers of $h$. As in the normal pseudodifferential calculus an operator $P$ acting on $L^{2}$ functions $u$ is given by its symbol $p(x,\xi,h)$ and a quantisation procedure.

\begin{defin}
Let $p(x,\xi,h)\in{}C^{\infty}(T^{\star}\R^{n})$ be a symbol in the symbol space $S^{m}$. We define the left semiclassical quantisation $p(x,hD)$ as
$$p(x,hD)u(x)=\frac{1}{(2\pi{}h)^{n}}\int{}e^{\frac{i}{h}<x-y,\xi>}p(x,\xi,h)u(y)dyd\xi$$
and the Weyl semiclassical quantisation $p^{w}(x,hD)$ as
$$p^{w}(x,hD)u(x)=\frac{1}{(2\pi{}h)^{n}}\int{}e^{\frac{i}{h}<x-y,\xi>}p\left(\frac{x+y}{2},\xi,h\right)u(y)dyd\xi.$$
\end{defin}

\begin{remark}
For real symbols the Weyl quantisation $p^{w}(x,hD)$ is self-adjoint. For this reason it will sometimes be more convenient to use $p^{w}(x,hD)$ in place of $p(x,hD)$.
\end{remark}

\begin{defin}\label{local}
A function $u$ depending parametrically on $h$ is said to satisfy the localisation condition if there exists $\chi\in{}C_{c}(T^{\star}M)$ such that
$$u=\chi(x,hD)u+O_{\mathcal{S}}(h^{\infty}).$$
where  ${\mathcal S}$ is the space of Schwartz functions.
\end{defin} 

This assumption allows us to move from a global problem to a local one. As $\chi$ has compact support in $T^{\star}M$ we can write
$$\chi(x,\xi)=\sum_{i=1}^{N}\chi_{i}(x,\xi)$$
for some $N<\infty$ where each $\chi_{i}$ has arbitarily small support. As noted previously the notion of an approximate solution is preserved under such a localisation. Now we may assume that we are working on a coordinate patch of $T^{\star}M$. Therefore we identify $M$ with $\R^{n}$ and $Y$ with $\R^{k}$. An element $x\in{}M$ will be denoted $x=(y,z)$ where $Y=\{z=0\}$. An element $\xi\in{}T_{x}^{\star}M$ will be written as $\xi=(\xi_{y},\xi_{z})$. Note that if $M$ is a compact manifold the localisation requirement in the spatial variables is trivially satisfied.

As we assume that $p(x,\xi,h)$ is smooth in $h$ we can write $p(x,\xi,h)=p_{0}(x,\xi)+hq(x,\xi,h)$. Now as $u$ is localised,
$$P(x,hD)\chi(x,hD)u=O_{L^{2}}(h)\Rightarrow{}P_{0}(x,hD)\chi(x,hD)u=O_{L^{2}}(h).$$
For the rest of this paper we will therefore assume that we are working with a symbol $p(x,\xi)$ independent of $h$.

Using the localisation assumption we are able to get a bound on $\norm{u}_{L^{p}}$ in terms of $\norm{u}_{L^{q}}$ where $q<p$. We have

$$u=\chi(x,hD)u+O(h^{\infty})$$
$$=h^{-n}\int{}K\left(x,\frac{x-y}{h}\right)u(y)dy+O(h^{\infty})$$
where
$$K(x,z)=\frac{1}{(2\pi)^{n}}\int\chi(x,\xi)e^{i<z,\xi>}d\xi.$$
A bound of $c(1-|x-z|)^{-N}$ is found for $|K(x,z)|$  by repeated integration by parts and by then applying Young's inequality the following estimate is obtained.

\begin{lem}\label{semiclassicalsobolev}
Suppose that a family $u=u(h)$ satisfies the localisation condition then for $1\leq{}q\leq{}p\leq\infty$
$$\norm{u}_{L^{p}}\lesssim{}h^{n(1/p-1/q)}\norm{u}_{L^{q}}+O(h^{\infty}).$$

\end{lem}

In a couple of places we will want to use this estimate over a submanifold rather than the full manifold. To do this we require localisation to hold if some variables are fixed.

\begin{lem}
If $u$ satisfies the localisation conditions then there exists some $\tilde{\chi}(y,{\xi_{y}})$ compactly supported such that 
$$R_{Y}u=\tilde{\chi}(y,hD_{y})R_{Y}u+O(h^{\infty})$$
where $R_{Y}$ is the restriction operator onto the submanifold $Y$.

\end{lem}

\begin{proof}
First as $u(x)$ is localised we can replace $u(x)$ with $\chi(x,hD)u$. Let $\psi(y,\xi_{y})\in{}C_{c}^{\infty}(\R^{k}\times\R^{k})$ such that $\psi(y,{\xi_{y}})=1$ for all $(y,\xi_{y})$ such that $(y,z,{\xi_{y}},{\xi_{z}})\in\mbox{Supp}\chi(x,\xi)$ for some $(z,{\xi_{z}})$. As $\chi(x,\xi)=\psi(y,{\xi_{y}})\chi(x,\xi)$ repeated application of non-stationary phase gives

$$(Id-\psi(y,hD_{y}))R_{Y}\chi(x,hD)u=O(h^{\infty})$$
which gives
$$R_{Y}\chi(x,hD)u=\psi(y,hD_{y})R_{Y}\chi(x,hD)u+O(h^{\infty})$$
as required.

\end{proof}

 Using this localisation condition we can prove that when $p(x,\xi)$ is bounded away from zero the local contribution is small. From (\cite{koch}, Lemma 2.1) we have that if $|p(x,\xi)|\geq{1/C}$ on a local patch then we can invert $p(x,hD)$ up to order $h^{\infty}$. That is, choosing $\chi(x,\xi)$ supported on this patch, we can find some $q(x,hD)$ such that
 $$q(x,hD)p(x,hD)\chi(x,hD)=\chi(x,hD)+O_{L^{2}\rightarrow{}L^{2}}(h^{\infty})$$
 and
 $$p(x,hD)q(x,hD)\chi(x,hD)=\chi(x,hD)+O_{L^{2}\rightarrow{}L^{2}}(h^{\infty}).$$
 So if $p(x,hD)u=O_{L^{2}}(h)$ and $|p(x,\xi)|>1/C$ we can invert $p(x,hD)$ to get
$$\chi(x,hD)u=O_{L^{2}}(h).$$
Now using Lemma \ref{semiclassicalsobolev} to estimate $\norm{u}_{L^{\infty}}$ by $\norm{u}_{L^{2}}$ we have
\begin{equation}\norm{u}_{\infty}=O(h^{-(n-2)/2}).\label{goodinfinity}\end{equation}
To get the $L^{2}$ norm of the restriction of $u$ to $Y$ we use Lemma \ref{semiclassicalsobolev} again this time only in the $z$ coordinates. We have
\begin{equation}\norm{u(y,0)}_{L^{2}_{y}}\lesssim\norm{u(y,z)}_{L^{\infty}_{z}L^{2}_{y}}\lesssim{}h^{-\frac{n-k}{2}}\norm{u(y,z)}_{L^{2}_{z}L^{2}_{y}}.\label{goodL2}\end{equation}
So the $L^{2}$ norm of $u$ when restricted to a submanifold is $O(h^{-\frac{n-k-2}{2}})$. Interpolating between (\ref{goodinfinity}) and (\ref{goodL2}) gives us better $L^{p}$ estimates than those given by Theorem \ref{maintheorem}. Consequently we can ignore regions where $p(x,\xi)$ is bounded away from zero.

This reduces our problem to localising around points $(x_{0},\xi_{0})$ where $p(x_{0},\xi_{0})=0$. To proceed we need to place some non-degeneracy conditions on $p(x,\xi)$.
\begin{defin}\label{admissable}
A symbol $p(x,\xi)$ is admissible if it satisfies the following non-degeneracy conditions:
\begin{itemize}
\item[(A1)] for any pair $(x_{0},\xi_{0})$ such that $p(x_{0},\xi_{0})=0$, $\partial_{\xi}p(x_{0},\xi_{0})\neq{}0$
\item[(A2)] the second fundamental form on $\{\xi\mid{}p(x_{0},\xi)=0\}\subset{}T_{x_{0}}^{\star}M$ is positive definite.
\end{itemize}
\end{defin}
The first condition will be used to convert this problem into one regarding evolution operators. The second condition is needed for some later stationary phase estimates. The main result of this paper is below.

\begin{thm}\label{maintheorem}
Let $(M,g)$ be a smooth Riemannian manifold with no boundary and let $Y$ be a smooth embedded submanifold with dimension $k$. Let $u(h)$ be a family of $L^{2}$ normalised functions that satisfy $Pu=O_{L^{2}}(h)$ for $P$ a semiclassical operator with symbol $p(x,\xi)$. Assume further that $u$ satisfies the localisation property and that the symbol $p(x,\xi)$ is admissible. Then the $L^{p}$ norms restricted to $Y$ are:

$$\norm{u}_{L^{p}(Y)}\lesssim{}h^{-\delta(n,k,p)}$$
\begin{equation}\delta(n,n-1,p)=
\begin{cases}
\frac{n-1}{2}-\frac{n-1}{p},&\frac{2n}{n-1}\leq{}p\leq\infty\\
\frac{n-1}{4}-\frac{n-2}{2p},&2\leq{}p\leq\frac{2n}{n-1}
\end{cases}
\end{equation}
and for $k\neq{}n-1$
\begin{equation}
\delta(n,k,p)=
\begin{cases}
\frac{n-1}{2}-\frac{k}{p},\:{}2\leq{}p\leq\infty,\:(k,p)\neq{}(n-2,2).
\end{cases}
\end{equation}
For $k=n-2$ the $L^{2}$ estimate is
$$\norm{u}_{L^{2}(Y)}\lesssim{}h^{-1/2}\log^{1/2}(1/h).$$

\end{thm}
\begin{remark}
Apart from the log loss in the $(k,p)=(n-2,2)$ case these estimates are known to be sharp for Laplacian eigenfunctions as shown by Burq, G\'{e}rard and Tzetkov \cite{burq}.
\end{remark}

In proving the semiclassical version for the full manifold estimates both Koch, Tataru and Zworski \cite{koch} and Burq, G\'{e}rard and Tzvetkov \cite{burq2} used the assumption (A1) along with the implicit function theorem to write $p(x,\xi)$ as
\begin{equation}p(x,\xi)=e(x,\xi)(\xi_{1}-a(x,\xi)).\label{implicit}\end{equation}
Then by using $x_{1}$ as a time variable, $t$, they reduced the problem to studying the evolution equation
$$(hD_{t}-a(t,x',hD_{x'}))u=0\label{assevol2}.$$
An approximate propagator $U(t)$ for (\ref{assevol2}) can be written down as a Fourier Integral Operator. By proving a decay estimate on $\norm{U(s)^{\star}U(t)}_{L^{1}\to{}L^{\infty}}$ they were able to use Strichartz estimates to determine the mixed ``space-time'' norm. Using the Strichartz estimate for the pair $(p,p)$ they obtained an estimate on the $L^{p}$ norm for $p=\frac{2(n+1)}{n-1}$. From the localisation assumption and Duhamels principle they determined the $L^{\infty}$ estimate. All other $L^{p}$ estimates were obtained by interpolation between these points and the trivial $L^{2}$ bound.

We follow a similar procedure to find estimates for $\norm{u}_{L^{p}(Y)}$. As the $L^{\infty}$ estimate on the submanifold must be the same as over the full manifold we only need to find the $L^{2}$ norm and the $L^{p}$ norm given by the appropriate Strichartz estimates. 

We cannot however use this method immediately, as we do not know whether the time variable $t=x_{1}$ determined by (\ref{implicit}) remains a valid co-ordinate when restricted to the submanifold $Y$. For example, $t$ could be constant on $Y$. However the localisation property comes to our aid at this point and allows us to prove the required estimates (or better) when $t$ is constant on $Y$. This provides a natural division of the problem into two cases. In case one the time variable is constant on $Y$ and, given the symbol factorisation, the proof of Theorem \ref{maintheorem} follows easily from conservation of energy and localisation. In the second case, where time is a coordinate when restricted to $Y$ we need to use Strichartz estimates. Although the usual form of Strichartz estimates do not fit this problem we are able to modify the abstract Strichartz estimates for our use.

The usual statement of Strichartz estimates assumes $L^{2}$ boundedness. In this case our family of operators $W(t)$ will be determined from the the full evolution operator by a restriction of some spatial variables and therefore is not necessarily $L^{2}$ bounded. However in the Keel-Tao \cite{keel} picture of Strichartz estimates which we will use this unitarity does not matter. We need only to have a bound from which to  interpolate. Obviously having a different interpolation endpoint will somewhat change the relationship between the Strichartz pair $(r,p)$ and $n$.

As we have shown that areas where $|p(x,\xi)|>1/C$ make negligible contributions we can study $p(x,\xi)$ around the points $(x_{0},\xi_{0})$ where $p(x_{0},\xi_{0})=0$.

In Section \ref{symbolfactorisation} we will factorise the symbol to create an evolution equation and show that if $\partial_{\xi_{z}}p(x_{0},\xi_{0})\neq{}0$ the localisation condition is enough to prove Theorem \ref{maintheorem}. Section \ref{adjustedstrichartz} gives the necessary extension of the abstract Strichartz estimates and governing equation for the Strichartz pairs $(r,p)$. Section \ref{approxpropagator} uses a Fourier Integral Operator to represent the evolution operator $U(t)$ and obtains estimates for the restriction of $U(t)$ to the submanifold. Section \ref{theoremproof} uses the estimates from Section 4 with the adjusted Strichartz estimate to prove Theorem \ref{maintheorem}.

\section{Symbol Factorisation}\label{symbolfactorisation}
By assumption (A1) we have that when $p(x,\xi)=0$, then  $\frac{\partial{}p}{\partial\xi_{i}}\neq{}0$ for some $i$. By the implicit function theorem we can solve the equation $\xi_{i}=a(x,\xi')$ on $\{\xi\mid{}p(x_{0},\xi)=0\}$ and, on the support of $\chi$, we have
$$p(x,\xi)=e(x,\xi)(\xi_{i}-a(x,\xi'))$$
where $e(x_{0},\xi_{0})\neq{}0$. Now, as $u$ is a quasimode, 
$$e(x,hD)(D_{x_{i}}-a(x,hD_{x'}))u=O_{L^{2}}(h).$$
As $e(x_{0},\xi_{0})\neq{}0$ we can, locally, approximately invert $e(x,hD)$ so now we have that,
$$(hD_{x_{i}}-a(x,hD_{x'}))u=O_{L^{2}}(h).$$
We the study the associated homogeneus evolution equation
$$(hD_{x_{i}}+a(x,hD_{x'}))u=0$$
where the $x_{i}$ space variable is thought of as the ``time'' variable. If we can understand the properties of the evolution operator $U(t)$ we will then be able to use Duhamel's principle to obtain estimates for $u$.

As we are estimating the restriction of $u$ to a submanifold we want to study a restricted for of $U(t)$ defined by
$$W(t)=R_{Y}\circ{}U(t).$$
It is now important to determine whether our time variable is a ``$z$'' variable (ie $Y$ is contained in a single time slice) or a ``$y$'' variable (ie $Y$ is transverse to time slices). To deal with this we will split the proof of Theorem \ref{maintheorem} into two cases. Case 1, where $\partial_{\xi_{z}}p(x_{0},\xi_{0})\neq{}0$ (the easy case) is proved below. Case 2, $\partial_{\xi_{z}}p(x_{0},\xi_{0})={}0$ (the harder case) requires the use of abstract Strichartz estimates that allow for non-unitary energy bounds. 

\begin{proof}[Proof of Theorem \ref{maintheorem} in Case 1]
We will prove that if $\partial_{\xi_{z}}p(x_{0},\xi_{0})\neq{}0$ the $L^{p}$ estimates for $u$ are at least as good (and possibly better) than those given by Theorem \ref{maintheorem}. This assumption implies $\frac{\partial{}p}{\partial{}\xi_{z_{i}}}\neq{}0$ for some $i$; we assume $i=1$. We can therefore factorise the symbol as
$$p(x,\xi)=e(x,\xi)(\xi_{z_{1}}-a(x,{\xi_{y}},{\xi_{z}}'))$$
where $z=(z_{1},z')$. As $Pu=O_{L^{2}}(h)$
and $e(x_{0},\xi_{0})\neq{}0$ we can conclude that
\begin{equation}(hD_{{z_{1}}}-a(x,hD_{y},hD_{z'}))u=hf(x)\label{ntevol}\end{equation}
where $\norm{f}_{L^{2}}=O(1)$.
The associated homogeneous evolution equation is
\begin{equation}(hD_{z_{1}}-a(x,hD_{y},hD_{z'}))u=0.\label{homevol}\end{equation}
Now allowing the variable $z_{1}$ to act as a time variable we can find a propagator $U(t)$ that gives a solution for (\ref{homevol}). The solution operator $U(t)$ will be unitary on $L^{2}$.

Using Duhamel's principle and denoting $x'=(y,z')$ we write

\begin{equation}u(z_{1},x')=U(z_{1})u(0,x')+{i}\int_{0}^{z_{1}}U(z_{1}-s)f(s,x')ds.\label{duhamel}\end{equation}
Combining (\ref{duhamel}) with the conservation of $L^{2}$ mass for the homogeneous problem we have that if $u$ is $L^{2}$ normalised the the $L^{2}$ mass of $u$ on the hypersurface $H=\{x|z_{1}=0\}$ is of order one. We now use the localisation assumption along with semiclassical Sobolev estimates (Lemma \ref{semiclassicalsobolev}) to obtain an estimate for the $L^{2}$ on the submanifold $Y$.

$$\norm{u(y,0)}_{L^{2}_{y}}\leq\norm{u(y,0,z')}_{L^{\infty}_{z'}L^{2}_{y}}\lesssim{}h^{\frac{n-k-1}{2}}\norm{u(y,0,z')}_{L^{2}_{z'}L^{2}_{y}}$$
$$\lesssim{}h^{\frac{n-k-1}{2}}.$$
Which (apart from the hypersurface case where it is better) is the estimate we are looking for. 
\end{proof}

Therefore without loss of generality we will, for the rest of this paper, assume $\partial_{\xi_{z}}p(x_{0},\xi_{0})={}0$ (case 2) which by (A1) from Definition \ref{admissable} implies $\partial_{\xi_{y}}p(x_{0},\xi_{0})\neq{}0$. To prove the estimate in this case we use the same kind of symbol factorisation but this time $y_{1}$ will act as the time variable.

We will use a Fourier Integral Operator representation of $U(t)$ to obtain $L^{2}\to{}L^{2}$ and $L^{1}\to{}L^{\infty}$ bounds for $W(t)W^{\star}(s)$. We can then use the Strichartz estimates to get an estimate on
$$\left(\int\norm{W(t)u}_{p}^{r}dt\right)^{1/r}$$
 where $p=r$. However as we will be fixing some of the spatial variable at zero we cannot guarantee that $W(t)$ will still be unitary. To deal with this we need to make an adjustment to the abstract Strichartz estimates.

\section{Extended Strichartz estimates}\label{adjustedstrichartz}

Working with the Keel-Tao \cite{keel} formalism we have a family of operators $W(t)$ such that
$$W(t):H\rightarrow{}L^{2}(X)$$ for some Hilbert space $H$ and measure space $X$. When we apply this we will have $H=L^{2}(\R^{n-1})$ and $X=\R^{k-1}$. Note that $\R^{n-1}$ is a time slice in $M$ and $X=\R^{k-1}$ is a time slice in $Y$. The Strichartz assumptions modified to include a semiclassical parameter $h$ (see Koch-Tataru-Zworski \cite{koch} and Burq-G\'{e}rard-Tzvetkov \cite{burq2}) are that,
$$\norm{W(t)f}_{L^{2}}\leq{}C\norm{f}_{H}$$
and
$$\norm{W(t)W^{\star}(s)f}_{L^{\infty}(X)}\leq{}h^{-\mu}(h+|t-s|)^{-\sigma}\norm{f}_{1}.$$
This gives a mixed norm estimate of
$$\left(\int\norm{W(t)f}_{L^{p}}^{r}dt\right)^{1/r}\lesssim{}h^{-\frac{\mu}{r\sigma}}\norm{f}_{H}$$
where 
$$\frac{2}{r}+\frac{2\sigma}{p}=\sigma$$
and $(r,p)\neq(2,\infty)$.

We adjust these estimates by allowing the $L^{2}$ norm of $W(t)W^{\star}(s)f$ to have a bound of a similar form to the $L^{\infty}$ bound.

\begin{prop}
Let $W(t)$, $t\in\R$ be a family of operators $W(t):H\rightarrow{}L^{2}(X)$, where $H$ is a Hilbert space and $(X,dx)$ is a measure space. Assume that  $W(t)$ satisfies the estimates
\begin{itemize}
\item For all $t,s\in\R$ and $f\in{}L^{1}(X)$
\begin{equation}\norm{W(t)W^{\star}(s)f}_{L^{\infty}({X})}\lesssim{}h^{-\mu_{\infty}}(h+|t-s|)^{-\sigma_{\infty}}\norm{f}_{L^{1}(X)}\label{energy}\end{equation}

\item For all $t,s,\in\R$ and $f\in{}L^{2}(X)$
\begin{equation}\norm{W(t)W^{\star}(s)f}_{L^{2}(X)}\lesssim{}h^{-\mu_{2}}(h+|t-s|)^{-\sigma_{2}}\norm{f}_{L^{2}(X)}\label{decay}\end{equation}
\end{itemize}
then
\begin{equation}\left(\int\norm{W(t)f}_{L^{p}}^{r}dt\right)^{1/r}\lesssim{}h^{-\left(\frac{\mu_{\infty}-\mu_{2}}{r(\sigma_{\infty}-\sigma_{2})}+\frac{\sigma_{\infty}\mu_{2}-\sigma_{2}\mu_{\infty}}{2(\sigma_{\infty}-\sigma_{2})}\right)}\norm{f}_{H}\label{sticest}\end{equation}
for pairs of $(r,p)$, $2<r\leq\infty$, $2\leq{}p{}\leq\infty$ such that
\begin{equation}\frac{2}{r}+\frac{1}{p}(\sigma_{\infty}-\sigma_{2})=\frac{\sigma_{\infty}}{2}.\label{goveq}\end{equation}
\end{prop}

\begin{proof}
Following Keel-Tao \cite{keel} we will prove the bilinear form of the estimate
\begin{equation}\left|\iint\left<(W^{\star}(s)F(s),W^{\star}(t)G(t)\right>dsdt\right|\lesssim\norm{F}_{L^{r'}_{s}L^{p'}_{x}}\norm{G}_{L^{r'}_{t}L^{p'}_{s}}.\label{bilinearform}\end{equation}
Converting (\ref{energy}) and (\ref{decay}) into bilinear forms we have the estimates
$$\left|\left<W^{\star}(s)F(s),W^{\star}(t)G(t)\right>\right|\lesssim{}h^{-\mu_{\infty}}(h+|t-s|)^{-\sigma_{\infty}}\norm{F(s)}_{L^{1}}\norm{G(t)}_{L^{1}}$$
and
$$\left|\left<W^{\star}(s)F(s),W^{\star(t)}G(t)\right>\right|\lesssim{}h^{-\mu_{2}}(h+|t-s|)^{-\sigma_{2}}\norm{F(s)}_{L^{2}}\norm{G(t)}_{L^{2}}.$$
Interpolation between these estimates yields,
$$|\left<W^{\star}(s)F(s),W^{\star}(t)G(t)\right>|\leq{}h^{-\beta(p,\mu_{1},\mu_{2})}(h+|t-s|)^{-\beta(p,\sigma_{1},\sigma_{2})}\norm{F(s)}_{L^{p'}}\norm{G(t)}_{L^{p'}}$$
where $$\beta(
p,\sigma_{1},\sigma_{2})=\frac{2(\sigma_{2}-\sigma_{\infty})}{p}+\sigma_{\infty}.$$

We now use Hardy-Littlewood-Sobolev for the $t$ and $s$ integrations. This will give us the equation governing the relationship between $r$ and $p$. We have that
$$\int\int\frac{f(x)g(y)}{|x-y|^{\gamma}}dxdy\leq{}\norm{f}_{L^{q_{1}}}\norm{g}_{L^{q_{2}}}$$
for
$0<\gamma<n$, and
$$\frac{1}{q'_{1}}+\frac{1}{q'_{2}}=\frac{\gamma}{n}.$$
In this case we set $q_{1}=q_{2}=r'$ and $$\gamma=\frac{2(\sigma_{2}-\sigma_{\infty})}{p}+\sigma_{\infty}$$
so Hardy-Littlewood-Sobolev gives us
$$\frac{2}{r}=\frac{2(\sigma_{2}-\sigma_{\infty})}{p}+\sigma_{\infty}.$$
Rearranging this gives
$$\frac{2}{r}+\frac{1}{p}(\sigma_{\infty}-\sigma_{2})=\frac{\sigma_{\infty}}{2}$$
as the governing equation for these Strichartz estimates. Note that when $\sigma_{2}=\mu_{2}=0$ and $\sigma_{\infty}=\sigma$ this is just the original abstract Strichartz estimates governing equation
$$\frac{2}{r}+\frac{2\sigma}{p}=\sigma.$$
 Now we need to substitute the governing equation into the $h$ index. Doing this and working through the algebra we get that
$$\left(\int\norm{W(t)f}_{L^{p}}^{r}dt\right)^{1/r}\leq{}h^{-\left(\frac{\mu_{\infty}-\mu_{2}}{r(\sigma_{\infty}-\sigma_{2})}+\frac{\sigma_{\infty}\mu_{2}-\sigma_{2}\mu_{\infty}}{2(\sigma_{\infty}-\sigma_{2})}\right)}.$$
\end{proof}
Note that this simplifies considerably when $\mu_{1}=\sigma_{1}$ and $\mu_{2}=\sigma_{2}$, to become
$$\left(\int\norm{W(t)f}_{L^{q}}^{r}dt\right)^{1/r}\leq{}h^{-1/r}.$$

\begin{remark}
It is of course possible to further generalise these estimates by assuming $L^{q}$ bounds on $W(t)W^{\star}(s)$ for some $(q_{0},q_{1})$ rather than the usual $(2,\infty)$.
\end{remark}

\section{Approximate Propagator}\label{approxpropagator}

\begin{prop}\label{FIO}

Suppose $U(t):L^{2}(\R^{d})\rightarrow{}L^{2}(\R^{d})$ satisfies 
$$hD_{t}U(t)+A(t)U(t)=0,\quad{}U(0)=Id$$
where A(t) is a pseudodifferential operator such that the symbol principal symbol of $A(t)$ is real and has no dependence on $h$. Then there exists some $t_{0}>0$ independent of $h$ such that for $0\leq{}t\leq{}t_{0}$
$$U(t)u(\bar{x})=\frac{1}{(2\pi{}h)^{d}}\int\int{}e^{\frac{i}{h}(\phi(t,\bar{x},\eta)-w\cdot\eta)}b(t,\bar{x},\eta,h)u(w)dwd\eta+E(t)u(\bar{x})$$
where
$$\partial_{t}\phi(t,\bar{x},\eta)+a_{t}(\bar{x},\partial_{\bar{x}}\phi(t,\bar{x},\eta))=0,\quad{}\phi(0,\bar{x},\eta)=\bar{x}\cdot\eta$$
$$b(t,\bar{x},\eta,h)\in{}C^{\infty}_{c}(\R\times{}T^{\star}\R^{d}\times{}\R)\quad{}E(t)=O(h^{\infty}):S'\rightarrow{}S$$

\end{prop}
\begin{proof}
This is in fact the normal parametrix construction yielding the eikonal equation for the phase function. See \cite{evans} Section 10.2.
\end{proof}

Similar to the proof in case one we will use symbol factorisation to obtain
$$p(x,\xi)=e(x,\xi)(\xi_{y_{1}}-a(x,\xi')),$$
where $\xi'=(\xi'_{y},\xi_{z})$ and study the evolution equatioon
$$hD_{t}-a(x,hD_{y'},hD_{z})=0$$
(see Section \ref{theoremproof}). Here we are using $y_{1}$ as the time variable thus our coordinate $x$ is now decomposed as $x=(t,y',z)$. In the notation of Proposition \ref{FIO}, $A(t)=a(x,hD_{y'},hD_{z})$ and $\bar{x}=(y',z)$.

As on $\{\xi\mid{}p(x_{0},\xi)=0\}$, $\xi_{y_{1}}=a(x,\xi')$ the second fundamental form $h_{ij}$ is given by
$$h_{ij}=-\frac{\partial^{2}a}{\partial_{\xi'_{i}}\partial_{\xi'_{j}}}.$$
The non-degeneracy condition (A2) implies $h_{ij}$ is a positive definite matrix, therefore on a small enough patch $\partial^{2}_{\eta}a$ (where $\eta$ is the dual variable to $\bar{x}=(y',z)$) is also positive definite.
Recall that $W(t)=R_{Y}\circ{}U(t)$ so we have (for $d=n-1$)
$$W(t)f(x)=\frac{1}{(2\pi{}h)^{d}}\iint{}e^{\frac{i}{h}(\phi(t,(y',0),\eta)-w\cdot\eta)}b(t,y',\eta,h)u(w)dwd\eta$$
In what follows we will write $\phi(t,(y',0),\eta)=\phi(t,y',\eta)$ and for $\eta\in\R^{d}$ understand $\langle{}y',\eta\rangle=\langle{}(y',0),\eta\rangle$. All dashed variable are in $\R^{k-1}$ and all undashed variables are in $\R^{d}=\R^{n-1}$.

\begin{prop}
\label{stress}
If $W(t)$ is as above then it satisfies the estimates
$$\norm{W(t)W^{\star}(s)f}_{L^{\infty}}\leq{}h^{-\frac{n-1}{2}}(h+|t-s|)^{-\frac{n-1}{2}}\norm{f}_{1}$$
$$\norm{W(t)W^{\star}(s)f}_{L^{2}}\leq{}h^{-\frac{n-k}{2}}(h+|t-s|)^{-\frac{n-k}{2}}\norm{f}_{2}$$

\end{prop}

\begin{proof}
 First we get a $L^{\infty}$ bound on the Schwartz kernel of $W(t)W^{\star}(s)$. This result can be found in \cite{koch} but for convenience we repeat it here. Using the integral representation for $U(t)$ and the fact that $W(t)f$ is the restriction of $U(t)f$ to $Y$ we write $W(t)W^{\star}(s)f$ as,
 $$W(t)W^{\star}(s)f=\int{}W(t,s,y',v')f(v')dv'$$
 where
 $$W(t,s,y',v')=\frac{1}{(2\pi{}h)^{2d}}\int_{\R^{3d}}e^{\frac{i}{h}(\phi(t,y',\eta)-\phi(s,v',\zeta)-<w,\eta-\zeta>)}Bdwd\eta{}d\zeta,$$
$B=B(t,s,y',v',w,\eta,\zeta;h)\in{}S(1)\cap{}C_{c}^{\infty}(\R^{2+6d})$. To find an estimate for $|W(t,s,y',v')|$ we will use repeated applications of the stationary phase method. First we calculate the critical points in $w$ and $\zeta$ allowing us to perform the $(w,\zeta)$ integration. The phase function $\phi$ is stationary and non-degenerate at $\zeta=\eta$, $w=\partial_{\zeta}\phi(s,v',\zeta)$ and so the stationary phase method implies that
 $$W(t,s,y',v')=\frac{1}{(2\pi{}h)^{d}}\int_{\R^{d}}e^{\frac{i}{h}(\phi(t,y',\eta)-\phi(s,v',\eta))}B_{1}(t,s,y',v',\eta;h)d\eta.$$
Finally we must use stationary phase again to deal with the $\eta$ integration. From the initial condition on $\phi$ in the formulation of the parametrix we can write
 $$\phi(s,y',\eta)-\phi(s,v',\eta)=\langle{}y'-v',\eta\rangle+\langle{}y'-v',sF(s,y',v',\eta)\rangle$$
 and so defining the phase function $\tilde{\phi}$ by 
 $$\tilde{\phi}(t,s,y',v',\eta)=\phi(t,y',\eta)-\phi(s,v',\eta)$$
 we have that
 $$\tilde{\phi}(t,s,y',v',\eta)=(t-s)a(0,y',\eta)+\langle{}y'-v',\eta+sF(s,y',v',\eta)\rangle+O(t-s)^{2}.$$
 So the phase is stationary when
 \begin{equation}0=\partial_{\eta}\tilde{\phi}=(Id+s\partial_{\eta}F)(y'-v')+(t-s)(\partial_{\eta}a+O(t-s)).\label{stationary}\end{equation}
 When $s$ is small, $Id+s\partial_{\eta}F$ is invertible and this implies that at a critical point
 $$|y'-v'|=O(t-s).$$
 The Hessian is given by
 $$\partial_{\eta}^{2}\tilde{\phi}=s\partial_{\eta}^{2}\langle{}y'-v',F\rangle+(t-s)(\partial_{\eta}^{2}a+O(t-s))$$
 $$=(t-s)(\partial_{\eta}^{2}a+O(|t|+|s|)$$
 where $\partial^{2}_{\eta}a=\partial^{2}_{\eta}a(0,y',\eta)$.
Here we use the non-degeneracy of $\partial^{2}_{\eta}a$ to give that if $t$ and $s$ are sufficiently small then, at a critical point,
 \begin{equation}\partial_{\eta}^{2}\tilde{\phi}=(t-s)\Psi(y',v',t,s,\eta)\label{2deriv}\end{equation}
 where $\Psi(y',v',t,s,\eta)$ is an invertible matrix,
 $$\det\Psi(y',v',t,s,\eta)\geq{}c>0$$
 and the elements of $\Psi(y',v',t,s,\eta)$ are smooth in all variables. So for $|t-s|>Mh$ for some suitably large $M$ we can apply the stationary phase method to conclude that
 $$|W(t,s,y',v')|\leq{}\frac{C}{h^{d}}\left(1+\frac{|t-s|}{h}\right)^{-\frac{d}{2}}$$
 $$\lesssim{}h^{-\frac{d}{2}}(h+|t-s|)^{-\frac{d}{2}}.$$
 When $|t-s|<Mh$ we can use trivial estimates to show that
 $$|W(t,s,y',v')|\leq{}Ch^{-d}$$
 $$\lesssim{}h^{-\frac{d}{2}}(h+|t-s|)^{-\frac{d}{2}}.$$
 From these estimates we can obtain the necessary bounds on the $L^{1}\to{}L^{\infty}$ norm of $W(t)W(s)^{\star}$. We have,
 $$\norm{W(t)W(s)^{\star}}_{L^{1}\rightarrow{}L^{\infty}}\lesssim\mbox{esssup}|W(t,s,y,v)|\leq{}Ch^{-\frac{d}{2}}(h+|t-s|)^{-\frac{d}{2}}.$$
For the $L^{2}$ estimate we need to use the oscillations of $W(t,s,y',v')$ itself. First note that from the critical point equation (\ref{stationary}) we have that if $|y'-v'|\geq{}K|t-s|$ for some suitably large $K$, critical points cannot occur. In this case we can estimate $|W
(t,s,y',v')|$ by nonstationary phase obtaining
$$|W(t,s,y',v')|\lesssim{}h^{-d}\left(1+\frac{|y'-v'|}{h}\right)^{-N}$$
and as $|y'-v'|\geq{}K|t-s|$ 
$$|W(t,s,y',v')|\lesssim{}h^{-d}\left(1+\frac{|y'-v'|}{h}\right)^{-N}\left(1+\frac{|t-s|}{h}\right)^{-N}.$$
In view of this we split 
$$W(t,s,y',v')=W_{1}(t,s,y',v')+W_{2}(t,s,y',v')$$
where
$$W_{1}(t,s,y',v')=\zeta\left(\frac{|y'-v'|}{|t-s|}\right)W(t,s,y',v')$$
$$W_{2}(t,s,y',v')=\left(1-\zeta\left(\frac{|y'-v'|}{|t-s|}\right)\right)W(t,s,y',v')$$
and $\zeta(r):\R\to\R$ is a smooth cut off function
$$\zeta(r)=\begin{cases}
1&|r|\leq{}\frac{3}{2}K\\
0&|r|\geq{}2K\end{cases}.$$
We now have
$$W(t)W^{\star}(s)=(W(\tau)W^{\star}(s))_{1}+(W(\tau)W^{\star}(s))_{2}$$
where $(W(t)W^{\star}(s))_{i}$ is the operator with integral kernel $W_{i}(t,s,y',v')$. Now by Young's inequality
$$\norm{(W(t)W^{\star}(s))_{2}f}_{L^{2}}\lesssim{}h^{-d}\left(1+\frac{|t-s|}{h}\right)^{-N}\norm{f}_{L^{2}}\int{}\left(1+\frac{|y'|}{h}\right)^{-N}dy'$$
$$\lesssim{}h^{-(d-k+1)}\left(1+\frac{|t-s|}{h}\right)^{-N}\norm{f}_{L^{2}}$$
it therefore only remains to deal with $(W(t)W^{\star}(s))_{1}$.

When $|t-s|\geq{}Mh$ we have
 $$W_{1}(t,s,y',v')=\frac{e^{\frac{i}{h}\psi(t,s,y',v')}b(t,s,y',v')\zeta\left(\frac{|y'-v'|}{|t-s|}\right)}{h^{d/2}(h+|t-s|)^{d/2}}$$
 where 
 $$\psi(t,s,y',v')=\tilde{\phi}(t,s,y',v',\eta(y',v',t,s))$$
and $\eta(y',v',t,s)$ is determined by (\ref{stationary}) (the implicit function theorem guarantees that given $\tilde{\phi}_{\eta}=0$ we can solve for $\eta=\eta(y',v',t,s)$ due to (\ref{2deriv}) and the invertibility of $\Psi(y',v')$). To exploit these oscillations we square the $L^{2}$ norm of $(W(\tau)W^{\star}(s))_{1}f$ and use nonstationary phase methods. We therefore need derivative bounds (in $y'$) on $b(t,s,y',v')$. Lemma \ref{derivativebounds} gives us that bounds on $D_{y'}^{\beta}b(t,s,y',v')$ depend only on bounds on $D_{y'}^{\alpha}\eta(y',v',t,s)$, $0\leq|\alpha|\leq|\beta|$.

\begin{lem}\label{derivativebounds}
If
$$\int{}e^{i\lambda\phi(x,\xi)}a(x,\xi)d\xi$$
is an oscillatory integral localised around a ($x$ dependent) nondegenerate critical point $\varphi(x)=(\varphi_{1}(x),\ldots,\varphi_{d}(x))$ and
\begin{equation}|D_{x}^{\alpha}\varphi_{i}(x)|\lesssim{}\gamma^{|\alpha|}\label{lemmaass}\end{equation}
for any multi-index $\alpha$ and $1\leq{}i\leq{}n$ then
$$\int{}e^{i\lambda\phi(x,\xi)}a(x,\xi)d\xi=(1+\lambda)^{\frac{d}{2}}e^{i\lambda{}\phi(x,\varphi(x))}b(x)$$
where $b(x)$ obeys the bounds
$$|D_{x}^{\beta}b(x)|\lesssim{}\gamma^{|\beta|}$$
for any multi-index $\beta$.

\begin{proof}

The method of stationary phase gives us the representation of
$$\int{}e^{i\lambda\phi(x,\xi)}a(x,\xi)d\xi=(1+\lambda)^{\frac{d}{2}}e^{i\lambda{}\phi(x,\varphi(x))}b(x)$$
so it remains only to check the derivative bounds of $b(x)$. We write 
$$b(x)=(1+\lambda)^{-\frac{d}{2}}\int{}e^{i\lambda\tilde{\phi}(x,\xi)}a(x,\xi)d\xi$$
where $\tilde{\phi}(x,\xi)=\phi(x,\xi)-\phi(x,\varphi(x))$. From the Morse lemma it is enough to prove the bounds for
$$\tilde{\phi}(x,\xi)=\sum_{j=1}^{m}(\xi_{j}-\varphi_{j}(x))^{2}-\sum_{j=m}^{n}(\xi_{j}-\varphi_{j}(x))^{2}.$$
In this case we have 
$$\frac{\partial\tilde{\phi}(x,\xi)}{\partial{}x_{i}}=-2\sum_{j=1}^{m}(\eta_{j}-\varphi_{j})\frac{\partial\varphi_{j}(x)}{\partial{}x_{i}}+2\sum_{j=m}^{n}(\eta_{j}-\varphi_{j})\frac{\partial\varphi_{j}(x)}{\partial{}x_{i}}$$
so
\begin{equation}\frac{\partial\tilde{\phi}(x,\xi)}{\partial{}x_{i}}=-[\partial_{x_{i}}\varphi(x)]^{T}\nabla_{\xi}\tilde{\phi}(x,\xi)\label{derivid}\end{equation}
where
$$[\partial_{x_{i}}\varphi(x)]=\begin{bmatrix}
\partial_{x_{i}}\varphi_{1}(x)\\
\vdots\\
\partial_{x_{i}}\varphi_{n}(x)\end{bmatrix}.$$
Now let
$$I(\lambda,a)=\int{}e^{i\lambda\tilde{\phi}(x,\xi)}a(x,\xi)d\xi$$
$$\frac{\partial}{\partial{}x_{i}}I(\lambda,a)=\int{}e^{i\lambda\tilde{\phi}(x,\xi)}\left(i\lambda\frac{\partial\varphi(x)}{\partial{}x_{i}}a(x,\xi)+\frac{\partial{}a(x,\xi)}{\partial{}x_{i}}\right)d\xi$$
By (\ref{derivid}) we have
$$\int{}e^{i\lambda\tilde{\phi}(x,\xi)}i\lambda\frac{\partial\varphi(x)}{\partial{}x_{i}}a(x,\xi)d\xi=-[\partial_{x_{i}}\varphi(x)]^{T}\int{}e^{i\lambda\tilde{\phi}(x,\xi)}i\lambda\nabla_{\xi}\tilde{\phi}(x,\xi)a(x,\xi)$$
and integrating by parts
$$\int{}e^{i\lambda\tilde{\phi}(x,\xi)}i\lambda\frac{\partial\varphi(x)}{\partial{}x_{i}}a(x,\xi)d\xi=[\partial_{x_{i}}\varphi(x)]^{T}\int{}e^{i\lambda\tilde{\phi}(x,\xi)}\nabla_{\xi}a(x,\xi)d\xi$$
so
$$\frac{\partial}{\partial{}x_{i}}I(\lambda,a)=I(\lambda,[\partial_{x_{i}}\varphi]^{T}\nabla_{\xi}a+\partial_{x_{i}}a).$$
The bounds on the derivatives of $b(x)$ therefore follow from (\ref{lemmaass}) and stationary phase estimates.

\end{proof}

\end{lem}

To obtain derivative bounds on $\eta(y',v',t,s)$ we differentiate (\ref{stationary}) in $y'$ to obtain
$$0=(\widetilde{Id}_{k-1,d}+O(|t|+|s|))+[\partial_{y'}\eta(y',v',t,s)]^{T}(t-s)\Psi(y',v')$$
where $\widetilde{Id}_{k-1,d}$ is the $(k-1,d)$ dimensional matrix $[Id_{k-1}|0]$. Therefore
\begin{equation}(t-s)[\partial_{y'}\eta(y',v',t,s)]^{T}=-(\widetilde{Id}_{k-1,d}+O(|t|+|s|))\Psi(y',v')^{-1}.\label{etaeq}\end{equation}
This gives us
$$|D_{y'_{i}}\eta(y',v',t,s)|\lesssim{}\frac{1}{|t-s|}$$
to obtain the multi-index bound we simply differentiate (\ref{etaeq}) leading to
$$|D_{y'}^{\alpha}\eta(y',v',t,s)|\lesssim{}\left(\frac{1}{|t-s|}\right)^{|\alpha|}$$
for any multi-index $\alpha$. So by Lemma \ref{derivativebounds} with $\lambda=|t-s|/h$ and $\gamma=|t-s|^{-1}$
$$|D^{\beta}_{y'}b(t,s,y',v')|\lesssim{}\left(\frac{1}{|t-s|}\right)^{|\beta|}$$
for any multi-index $\beta$.

Now we have
 $$\norm{(W(t)W^{\star}(s))_{1}f}_{L^{2}}^{2}=\iiint{}W_{1}(t,s,y',v')\overline{W}_{1}(t,s,y',w')f(v')\bar{f}(w')dv'dw'dy'$$
 $$=\iint{}\widetilde{W}(t,s,v',w')f(v')\bar{f}(w')dv'dw'$$
 where
 \begin{equation}\widetilde{W}(t,s,v',w')=\int{}W_{1}(t,s,y',v')\overline{W}_{1}(t,s,y',w')dy'.\label{L2kernel}\end{equation}
We will estimate (\ref{L2kernel}) via non-stationary phase estimates. The phase function in question is
$$\psi(t,s,y',v')-\psi(t,s,y',w').$$
From Taylor's theorem we have that
$$\nabla_{y_{i}'}[\psi(t,s,y',v')-\psi(t,s,y',w')]=\sum_{j=1}^{k-1}\frac{\partial^{2}\psi}{\partial{}y'_{i}\partial{}v'_{j}}(v'-w')+O(|v'-w'|^{2}),$$
written in matrix form this is
$$\nabla_{y'}[\psi(t,s,y',v')-\psi(t,s,y',w')]=\frac{\partial^{2}\psi}{\partial{}y'\partial{}v'}(v'-w')+O(|v'-w'|^{2}).$$
So we study the matrix $\frac{\partial^{2}\psi}{\partial{}y'\partial{}v'}$. As
$$\partial_{\eta}\tilde{\phi}(t,s,y',v',\eta(y',v',t,s))\equiv{}0$$
and
$$\tilde{\phi}(t,s,y',v',\eta)=\phi(t,y',\eta)-\phi(s,v',\eta)\Rightarrow\frac{\partial^{2}\tilde{\phi}}{\partial{}y'_{i}\partial{}v'_{j}}=0$$
we get
$$\frac{\partial^{2}\psi}{\partial{}y'_{i}\partial{}v'_{j}}=-\sum_{k,l=1}^{d}[\partial_{v_{j}'}\eta_{k}(y',v',t,s)]\left(\partial_{\eta_{k}}\partial_{\eta_{l}}\tilde{\phi}\right)[\partial_{y'_{i}}\eta_{l}(y',v',t,s)]$$
for $i,j=1\dots{}k-1$. In matrix form this is
$$\frac{\partial^{2}\psi}{\partial{}y'\partial{}v'}=[\partial_{v'}\eta(y',v',t,s)]^{T}\partial_{\eta}^{2}\tilde{\phi}[\partial_{y'}\eta(y',v',t,s)].$$
We already have
$$\partial_{\eta}^{2}\tilde{\phi}=(t-s)\Psi(y',v')=(t-s)(\partial^{2}_{\eta}a+O(|t|+|s|))$$
and
$$(t-s)[\partial_{y'}\eta(y',v',t,s)]^{T}=-(\widetilde{Id}_{k-1,d}+O(|t|+|s|))\Psi(y',v')^{-1}.$$
so we only need an expression for $[\partial_{v'}\eta(y',v',t,s)]^{T}$. Differentiating (\ref{stationary}) in $v'$ gives
$$0=-(\widetilde{Id}_{k-1,d}+O(|t|+|s|))+[\partial_{v'}\eta(y',v',t,s)]^{T}(t-s)\Psi(y',v')$$
so
$$(t-s)[\partial_{v'}\eta(y',v',t,s)]^{T}=-(\widetilde{Id}_{k-1,d}+O(|t|+|s|))\Psi(y',v')^{-1}.$$
Therefore
$$[\partial_{v'}\eta(y',v't,s)]^{T}\partial_{\eta}^{2}\tilde{\phi}[\partial_{y'}\eta(y',v',t,s)]=\frac{-1}{t-s}\left(\widetilde{Id}_{k-1,d}(\partial^{2}_{\eta}a)^{-1}\widetilde{Id}_{k-1,d}^{T}+O(|t|+|s|)\right).$$
The leading term is the upper $(k-1,k-1)$ block matrix of $\partial_{\eta}^{2}a$. As $\partial_{\eta}^{2}a$ is positive definite the matrix 
$\frac{\partial^{2}\psi}{\partial{}y'\partial{}v'}$ is non-degenerate. Consequently
$$\left|\nabla_{y'}[\psi(t,s,y',v')-\psi(t,s,y',w')]\right|\geq{}\frac{c|v'-w'|}{|t-s|}.$$
Therefore any integration by parts of (\ref{L2kernel}) will gain a factor of
$$\frac{h|t-s|}{|v'-w'|}.$$  
However each integration by parts also gains a factor of
$$\frac{1}{|t-s|}$$
 from differentiating the symbol. Overall each integration by parts gains
$$\frac{h}{|v'-w'|}$$
So we have a bound on $\widetilde{W}(t,s,v',w')$ of
$$|\widetilde{W}(t,s,v',w')|\lesssim{}h^{-d}|t-s|^{-d}\left(1+\frac{|v'-w'|}{h}\right)^{-N}\int{}\zeta\left(\frac{|y'-v'|}{|t-s|}\right)\zeta\left(\frac{|y'-w'|}{|t-s|}\right)dy'$$
$$\lesssim{}h^{-d}|t-s|^{-(d-k+1)}\left(1+\frac{|v'-w'|}{h}\right)^{-N}$$
and
$$\norm{(W(t)W^{\star})_{1}(s)f}_{L^{2}}^{2}\lesssim{}h^{-d}|t-s|^{-(d-k+1)}\iint\frac{f(v')\bar{f}(w')dv'dw'}{\left(1+\frac{|v'-w'|}{h}\right)^{N}}$$
for all $N>0$. Therefore by Holder and Young
$$\norm{(W(t)W^{\star}(s))_{1}f}_{L^{2}}^{2}\lesssim{}h^{-d}|t-s|^{-(d-k+1)}h^{k-1}\norm{f}_{L^{2}}^{2}$$
$$\lesssim{}h^{-(d-k+1)}|t-s|^{-(d-k+1)}\norm{f}^{2}_{L^{2}}.$$
So for $|t-s|\geq{}Mh$
$$\norm{(W(t)W^{\star}(s))_{1}f}_{L^{2}}\lesssim{}h^{-\frac{d-k+1}{2}}|t-s|^{-\frac{d-k+1}{2}}\norm{f}_{L^{2}}.$$
It now remains to deal with the case $|t-s|\leq{}Mh$. This can be achieved by scaling. In this case $W_{1}(t,s,y',v')$ is only supported on the region $|y'-v'|\lesssim{}h$. We have that
 $$|W(t,s,y',v')|\lesssim{}h^{-d}.$$
Using Young's inequality we obtain
$$\norm{(W(t)W^{\star}(s))_{1}f}_{L^{2}}\lesssim{}h^{-(d+k-1)}\norm{f}_{L^{2}}.$$ Hence
 $$\norm{(W(t)W^{\star}(s))_{1}f}_{L^{2}}\leq{}Ch^{-\frac{d-k+1}{2}}(h+|t-s|)^{-\frac{d-k+1}{2}}\norm{f}_{L^{2}}.$$
 Putting this together with the estimates we already had for $(W(t)W^{\star}(s))_{2}$ we obtain
 $$\norm{W(t)W^{\star}(s)f}_{L^{2}}\leq{}Ch^{-\frac{d-k+1}{2}}(h+|t-s|)^{-\frac{d-k+1}{2}}\norm{f}_{L^{2}}.$$
 As we have used one of our original spatial variables as time we have $d=n-1$. This completes the proof.
 
 \end{proof}
 
\begin{remark}\label{posdef}
In these submanifold cases it is not enough to assume, as Koch-Tataru-Zworski \cite{koch} did in the full manifold case, that the second fundamental form on $\{\xi\mid{}p(x_{0},\xi)=0\}$ is merely non-degenerate. This  would imply that $\partial_{\eta}^{2}a$ is non-degenerate, however that is not enough to guarantee that the upper $(k-1,k-1)$ block matrix of $\partial_{\eta}^{2}a$ is also non-degenerate. Therefore we cannot prove the $L^{2}\to{}L^{2}$ estimates on $W(t)W^{\star}(s)$ if we assume only non-degeneracy. Note that the $L^{1}\to{}L^{\infty}$ estimate does however still hold under the weaker assumption of non-degeneracy.
\end{remark}

 We can now use Strichartz estimates (Proposition \ref{stress}) on $W(t)$.  We are in the case that $\mu_{1}=\sigma_{1}$ and $\mu_{2}=\sigma_{2}$, so we have 
 $$\left(\int\norm{W(t)f}_{L^{p}}^{r}dt\right)^{1/r}\lesssim{}h^{-1/r}\norm{f}_{L^{2}}$$
 when
 $$\frac{1}{r}+\frac{k-1}{2p}=\frac{n-1}{4}.$$
So this gives us that when $r=p$

$$p=\frac{2(k+1)}{n-1}.$$
In particular for $k=d-1$,

$$p=\frac{2n}{n-1}.$$
When $k=n-2$, $p=2$ so this is an endpoint. When $k\leq{}d-3$, $\frac{2(k+1)}{n-1}<2$ so the Strichartz estimates give us no point $(p,p)$.

\section{Completion of Proof in Case 2}\label{theoremproof}

Recall that Case 2 was $\partial_{\xi_{z}}p(x_{0},\xi_{0})=0$ and so by (A1) in Definition \ref{admissable} $\partial_{\xi_{y}}p(x_{0},\xi_{0})\neq{}0$. Without loss of generality we assume $\partial_{\xi_{y_{1}}}p(x_{0},\xi_{0})\neq{}0$. Around the point $(x_{0},\xi_{0})$ where $p(x_{0},\xi_{0})=0$ we use $\partial_{\xi_{y}}p(x_{0},\xi_{0})\neq{}0$ and the implicit function theorem to factorise $p(x,\xi)$ as
$$p(x,\xi)=e(x,\xi)(\xi_{y_{1}}-a(x,\xi')).$$
So $Pu=O_{L^{2}}(h)$ implies
$$e(x,hD)(hD_{y_{1}}-a(x,hD_{y'},hD_{z})u=O(h).$$
As $e(x,hD)$ is elliptic this implies
$$(hD_{y_{1}}-a(x,hD_{y'},hD_{z})u=hf(y_{1},x')$$
where $\norm{f}_{L^{2}(M)}=O_{L^{2}}(1)$.

Using Duhamel's principle we write

$$u(y_{1},x')=U(y_{1})u(0,x')+i\int_{0}^{y_{1}}U(y_{1}-s)f(s,x')ds.$$
When we restrict to the submanifold $Y$ by setting $z=0$ we get

$$u(y_{1},y',0)=W(y_{1})u(0,x')+i\int_{0}^{y_{1}}W(y_{1}-s)f(s,x')ds.$$

As we already have the $L^{\infty}$ estimates we are looking for a bound for the $L^{2}$ norm and the the bound given by the Strichartz estimates where appropriate.  Using Minkowski's inequality we have for any $q$ 

\begin{multline}\norm{u}_{L^{q}(Y)}\lesssim\left(\int\norm{W(y_{1})u_{0}}_{L^{q}_{y'}}^{q}dy_{1}\right)^{1/q}+\\
\int_{\R}\left(\int\norm{W(y_{1}-s)f(s,x')}_{L^{q}_{y'}}^{q}dy_{1}\right)^{1/q}ds\label{qnorm}\end{multline}
where $u_{0}=u(0,x')$.
Therefore to obtain a $L^{q}$ bound we need to estimate
$$\left(\int\norm{W(t)u_{0}}_{L_{y'}^{q}}^{q}dt\right)^{1/q}.$$

In the case where $k=n-1$ we obtain an estimate from Strichartz, see proposition \ref{stress}. Applying adjusted form of Strichartz estimates with $p=\frac{2n}{n-1}$ we have

$$\norm{u}_{L^{p}(Y)}\lesssim{}h^{-1/p}\norm{u_{0}}_{L_{x'}^{2}}+h^{-1/p}\int_{\R}\norm{f(s,x')}_{L^{2}_{x'}}ds$$
$$\lesssim{}h^{-1/p}.$$
For all other $k$ either there is no pair $(p,p)$ given by the Strichartz estimates or the pair is the endpoint pair $(2,2)$. 

We also need to obtain the $L^{2}$ estimates. These can be obtained directly from the bilinear form (\ref{bilinearform}).
\begin{prop}
The following submanifold estimates hold

$$\norm{u}_{L^{2}(Y)}\lesssim
\begin{cases}
h^{-\frac{n-k-1}{2}},&k\leq{}n-3\\
h^{-1/4},&k=n-1\end{cases}.$$
For $k=n-2$
$$\norm{u}_{L^{p}(Y)}\lesssim{}
\begin{cases}
h^{-\left(\frac{n-1}{2}-\frac{n-2}{p}\right)},&p>2\\
\log^{1/2}(1/h)h^{-1/2},&p=2\end{cases}.$$
\end{prop}

\begin{proof}

We will determine these bounds directly from the estimates on the bilinear forms. We have that if
$$\iint\left|\langle{}W^{\star}(s)F(s),W^{\star}(t)G(t)\rangle\right|\lesssim{}h^{-\delta}\norm{F}_{L^{2}_{t}L^{2}_{x}}\norm{G}_{L_{t}^{2}L_{x}^{2}}$$
then
$$\left(\int\norm{W(t)f}_{L^{2}(X)}^{2}dt\right)^{1/2}\lesssim{}h^{-\delta/2}\norm{f}_{H}.$$
Therefore using the estimate determined in Proposition (\ref{stress}) we need to get an estimate on

$$h^{-\frac{n-k}{2}}\iint\frac{\norm{F(s)}_{2}\norm{G(t)}_{2}}{(h+|t-s|)^{\frac{n-k}{2}}}dsdt$$
which by H\"{o}lder is the same as estimating
$$h^{-\frac{n-k}{2}}\norm{(h+|t|)^{-\frac{n-k}{2}}\star\norm{F}_{2}}_{L^{2}_{t}}\norm{G}_{L^{2}_{t}L^{2}_{x}}.$$
Using Young's inequality this reduces to estimating
$$\norm{(h+|t|)^{-\frac{n-k}{2}}}_{L^{1}_{t}}.$$
As we are on a compact manifold and the ``time'' variable is actually one of our space variables this corresponds to estimating
$$h^{-\frac{n-k}{2}}\int_{0}^{C}(h+\tau)^{-\frac{n-k}{2}}d\tau.$$
Pulling the $h$ out of the denominator and making a change of variable gives means this is equivalent to estimating
$$\lesssim{}h^{-\frac{n-k}{2}}\cdot{}h^{-\frac{n-k}{2}}\cdot{}h\int_{0}^{C/h}(1+\sigma)^{-\frac{n-k}{2}}d\sigma$$
$$\lesssim{}h^{-(n-k-1)}\int_{0}^{C/h}(1+\sigma)^{-\frac{n-k}{2}}d\sigma.$$
When $k\leq{}d-3$ the integral is $O(1)$ therefore
$$\left(\int\norm{W(t)u}_{L^{2}}^{2}dt\right)^{1/2}\lesssim{}h^{-\frac{n-k-1}{2}}.$$
Substituting this into (\ref{qnorm}) we get that

$$\norm{u}_{L^{2}(Y)}\lesssim{}h^{-\frac{n-k-1}{2}}\left(1+\int_{\R}\norm{f(s,x')}_{L^{2}_{x'}}ds\right)$$
$$\lesssim{}h^{-\frac{n-k-1}{2}}.$$
When $k=n-1$ we  estimate

$$\int_{0}^{C/h}(1+\sigma)^{-1/2}d\sigma=\left[(1+\sigma)^{1/2}\right]^{C/h}_{0}$$\
$$\lesssim{}h^{-1/2}.$$
So

$$\left(\int\norm{W(t)u}_{L^{2}}^{2}dt\right)^{1/2}\lesssim{}h^{-1/4}.$$
Again substituting this estimate into (\ref{qnorm}) gives
$$\norm{u}_{L^{2}(Y)}\lesssim{}h^{-1/4}.$$
For $k=n-2$
$$\int_{0}^{C/h}(1+\sigma)^{-\frac{n-k}{2}}\lesssim\log(1/h),$$
so
$$\norm{u}_{L^{2}(Y)}\lesssim\log^{1/2}(1/h)h^{-1/2}.$$
For $p>2$ we estimate
$$h^{-\beta(p,\frac{n-1}{2},1)}\iint\frac{\norm{F(s)}_{L^{p'}}\norm{G(t)}_{L^{p'}}}{(h+|t-s|)^{\beta(p,\frac{n-1}{2},1)}}dsdt$$
by applying H\"{o}lder and then Young we have

$$h^{-\beta(p,\frac{n-1}{2},1)}\left(h^{-\frac{p}{2}\beta(p,\frac{n-1}{2},1)}h\int_{0}^{C/h}(1+\sigma)^{-\frac{p}{2}\beta(p,\frac{n-1}{2},1)}d\sigma\right)^{\frac{2}{p}}.$$
When $p>2$ the integral is $O(1)$ therefore

$$\left(\int\norm{W(t)u}_{p}^{p}dt\right)^{1/p}\lesssim{}h^{\beta(p,\frac{n-1}{2},1)+1/p}$$
$$=h^{-\left(\frac{n-1}{2}-\frac{n-2}{p}\right)}$$
which implies the estimate

$$\norm{u}_{L^{p}_{Y}}\lesssim{}h^{-\left(\frac{n-1}{2}-\frac{n-2}{p}\right)}.$$

\end{proof}

We can now estimate the other $L^{p}$ norms by interpolation between these estimates thereby arriving at the full range of estimates. This completes the proof of Theorem \ref{maintheorem}.

\begin{remark}
As noted in Remark \ref{posdef} the $L^{1}\to{}L^{\infty}$ estimate on $W(t)W^{\star}(s)$ holds if we weaken condition (A2) in definition \ref{admissable} to require the second fundamental form on $\{\xi\mid{}p(x_{0},\xi)=0\}$ to be non-degenerate. From this estimate by Young and Hardy-Littlewood-Sobolev we can still obtain some estimates for small $k$ and large $p$. If $k<\frac{n-1}{2}$ the the full range of estimates hold. For $k\geq{}\frac{n-1}{2}$ we obtain the estimates given by Theorem \ref{maintheorem} if
$$
\begin{cases}
p\geq{}\frac{4k}{n-1}\quad{}k>\frac{n-1}{2}\\
p>\frac{4k}{n-1}\quad{}k=\frac{n-1}{2}\end{cases}.
$$
\end{remark}

\bibliography{Lpestimates1}
\bibliographystyle{plain}

\end{document}